\documentclass[reqno]{amsart}
\usepackage{amsmath,mathtools,amssymb}
\usepackage{color}

\mathtoolsset{showonlyrefs}

\usepackage{amsmath}
\usepackage{amsthm}
\usepackage{amsfonts}
\usepackage{amssymb}
\usepackage{mathrsfs}
\usepackage[shortlabels]{enumitem}
\usepackage{graphicx}
\usepackage{subfigure}
\usepackage{color}
\usepackage[dvipsnames]{xcolor}

\newtheorem{theorem}{Theorem}
\newtheorem{prop}[theorem]{Proposition}
\newtheorem{lemma}[theorem]{Lemma}

\theoremstyle{definition}

\newtheorem{rem}[theorem]{Remark}

\newcommand{\N}{\mathbb{N}}
\newcommand{\Z}{\mathbb{Z}}

\newcommand{\R}{\mathbb{R}}
\newcommand{\C}{\mathbb{C}}

\newcommand\e{\mathrm{e}}
\newcommand\I{\mathrm{i}}
\newcommand\re{\operatorname{Re}}
\newcommand\im{\operatorname{Im}}
\newcommand{\set}[2]{\{#1 : #2 \}}
\def\restrict{\upharpoonright}

\newcommand\dom{\mathcal D}
\newcommand{\Ran}{\mathop{\rm Ran}}

\newcommand\eps\varepsilon
\renewcommand\epsilon\varepsilon
\renewcommand\rho\varrho
\newcommand\lm\lambda

\newcommand{\dist}{\operatorname{dist}}

\newcommand{\rd}{\mathrm{d}}

\newcommand{\beq}{\begin{equation}}
\newcommand{\eeq}{\end{equation}}
\newcommand{\be}{\begin{equation*}}
\newcommand{\ee}{\end{equation*}}
\newcommand{\bmat}{\begin{pmatrix}}
\newcommand{\emat}{\end{pmatrix}}

\begin{document}
\title[Counterexample to the Laptev--Safronov conjecture]{Counterexample to the Laptev--Safronov conjecture}
\subjclass[2020]{35P15, 31Q12.}
\author[S. Bögli]{Sabine Bögli}
\address[S. Bögli]{Department of Mathematical Sciences, Durham University, Upper Mountjoy Campus,  Durham DH1 3LE, United Kingdom}
\email{sabine.boegli@durham.ac.uk}
 \author[J.-C.\ Cuenin]{Jean-Claude Cuenin}
 \address[J.-C.\ Cuenin]{Department of Mathematical Sciences, Loughborough University, Loughborough,
 Leicestershire LE11 3TU, United Kingdom}
 \email{J.Cuenin@lboro.ac.uk}

\date{\today}


\begin{abstract}
We prove that the Laptev--Safronov conjecture (Comm.\ Math.\ Phys.\, 2009) is false in the range that is not covered by Frank's positive result (Bull.\ Lond.\ Math.\ Soc.\, 2011). The simple counterexample is adaptable to a large class of Schrödinger type operators, for which we also prove new sharp upper bounds.  
\end{abstract}

\maketitle

\section{Introduction}

Consider a Schrödinger operator $H_V=-\Delta+V$ on $L^2(\R^d)$ with a complex-valued potential $V$. The Laptev--Safronov conjecture \cite{MR2540070} stipulates that in $d\geq 2$ dimensions any \emph{non-positive} eigenvalue $z$ of $H_V$ satisfies the bound
\begin{align}\label{LT bound}
|z|^{\gamma}\leq D_{\gamma,d}\int_{\R^d}|V(x)|^{\gamma+\frac{d}{2}}\rd x
\end{align}      
for $0<\gamma\leq d/2$, and with $D_{\gamma,d}$ independent of $V$ and $z$. It is known that the condition $\gamma\leq d/2$ is necessary, see \cite{MR3627408}. The inequality \eqref{LT bound} is known to be true if $d=1$ and $\gamma=1/2$ or if $d\geq 2$ and $\gamma\leq 1/2$. The one-dimensional bound (with the sharp constant $D_{1/2,1}=1/2$) is due to Abramov-Aslanyan-Davies \cite{AAD01}, and the higher dimensional bound is due to Frank \cite{MR2820160}. Originally, these bounds were stated for $z\in\C\setminus \R_+$, but it was later realized by Frank and Simon \cite{MR3713021} that embedded eigenvalues $z\in\R_+$ can also be accommodated. In fact, Frank and Simon ``almost disproved" the Laptev--Safronov conjecture by constructing a counterexample (based on an earlier example of Ionescu and Jerison \cite{MR2024415}) that prohibits \eqref{LT bound} for $z\in \R_+$ whenever $d\geq 2$ and $\gamma>1/2$. Here we prove that the Laptev--Safronov conjecture is \textit{false} in the form originally stated. In the following we write $q=\gamma+d/2$.
For any $\epsilon>0$, let $\chi_\epsilon$ be the indicator function of 
$$\{x=(x_1,x')\in\R\times\R^{d-1}:\,|x_1|<\eps^{-1},|x'|<\eps^{-1/2}\}.$$
We construct potentials $V_\eps$, $\eps>0$, with $|V_\eps|\leq \eps \chi_\eps$ and such that $z_\eps=1+i\eps$ is an eigenvalue of $H_{V_\eps}$. This allows us to disprove the conjecture.

\begin{theorem}\label{thm LS}
Let $d\geq 2$ and $q>(d+1)/2$. Then
\begin{align*}
\limsup_{\eps\to 0}\, \frac{|z_\eps|^{q-\frac{d}{2}}}{\|V_\eps\|_q^q}=+\infty.
\end{align*}
\end{theorem}
%
Actually, our counterexample shows more, namely that the following substitute of \eqref{LT bound} for ``long-range" potentials (i.e.\ $q>(d+1)/2$), due to Frank \cite{MR3717979}, 
\begin{align}\label{Frank bound long range}
\dist(z,\R_+)^{q-\frac{d+1}{2}}|z|^{\frac{1}{2}}\leq C_{q,d}\|V\|_q^q,
\end{align}  
is \emph{sharp} (in the sense that the exponent of $\dist(z,\R_+)$ cannot be made smaller while preserving scale-invariance). 
\begin{theorem}
Let $d\geq 2$ and $q\geq (d+1)/2$. Then
\begin{align*}
\liminf_{\eps\to 0}\, \frac{\dist(z_\eps,\R_+)^{q-\frac{d+1}{2}}|z_\eps|^{\frac{1}{2}}}{\|V_\eps\|_q^q}>0.
\end{align*}
\end{theorem}
In addition, the same example saturates the recent bound of the second author \cite[Th.\ 1.1]{MR4104544}, which states that
\begin{align}\label{DN higher dim}
|z|^{\frac{1}{2}}\leq C_d\sup_{y\in\R^d}\int_{\R^d}|V(x)|^{\frac{d+1}{2}}\exp(-E|x-y|)\rd x,
\end{align} 
with $E=\im\sqrt{z}>0$, and generalizes the one-dimensional analog of Davies and Nath \cite{MR1946184} to higher dimensions. Note that estimating the right hand side of \eqref{DN higher dim} from above by the same expression with $E=0$ results in the endpoint estimate \eqref{LT bound} with $\gamma=1/2$ (or equivalently $q=(d+1)/2$). Hence the endpoint case in Theorem~2 already implies that \eqref{DN higher dim} is optimal in some sense. 
The exponential factor in \eqref{DN higher dim} effectively localizes the integration to a ball $B(y,C/E)$.
Moreover, the right hand side of~\eqref{DN higher dim} may be much smaller than that of \eqref{Frank bound long range}.
There are estimates similar to \eqref{DN higher dim} for any $q\in (d/2,(d+1)/2]$, see~\cite{MR4104544}. We do not state these here but remark that our counterexample also shows that an analog of \cite[Th. 1]{MR4104544} cannot hold for $q>(d+1)/2$. For brevity, we denote the right hand side of~\eqref{DN higher dim} by $F_V(E)$. 
\begin{theorem}
Let $d\geq 2$. There exists $C_d'>0$ such that for all $L\geq 1$
\begin{align*}
\liminf_{\eps\to 0}\, \frac{|z_\eps|^{\frac{1}{2}}}{F_V(L\im\sqrt{z_\eps})}\geq C_d'\,L.
\end{align*}
\end{theorem}
These three theorems show that the bounds \eqref{LT bound},\eqref{Frank bound long range},\eqref{DN higher dim} provide a rather complete picture of sharp eigenvalue inequalities for Schrödinger operators with complex potentials. Some refinements for singular potentials are known, see e.g.\ \cite{MR2820160}, \cite{MR3865141}, \cite{MR3950666}, \cite{cassano2021eigenvalue}, but we focus here on the long-range aspects of the potential. 
This is reflected by the fact that the construction of our counterexample is local in Fourier space, similar to the examples for embedded eigenvalues in \cite{MR4107520}, where a connection between the aforementioned Ionescu--Jerison example and the ``Knapp example" in Fourier restriction theory (see e.g.\ \cite{MR1232192}, \cite{MR3243741} , \cite{MR3052498}, \cite{MR3971577} for textbook presentations) was made. The examples in \cite{MR4107520} are based on superpositions of infinitely many Knapp wavepackets, while our example here is based on a single such wavepacket.

As in \cite{MR4107520}, the locality in Fourier space affords the flexibility to adapt the counterexample to a large class of Schrödinger type operators of the form
\begin{align}\label{gen. Schroedinger op.}
H_V=h_0(D)+V(x),
\end{align}
where $h_0$ is a tempered distribution on $\R^d$, smooth in a neighborhood of some point $\xi^0\in\R^d$ and such that $\lambda:=h_0(\xi^0)$ is a regular value of $h_0$. This means that the isoenergy surface
\begin{align}\label{Fermi surface}
S_{\lambda}=\set{\xi\in\R^d}{h_0(\xi)=\lambda}
\end{align}
is a smooth nonempty hyersurface near $\xi^0$. 
Here $h_0(D)f=\mathcal{F}^{-1}(h_0\hat{f})$ is the Fourier multiplier corresponding to $h_0$ and $\mathcal{F}^{-1}$ is the inverse Fourier transform. It is well known that upper bounds for the resolvent $(H_0-z)^{-1}$, for $z$ close to $\lambda$, crucially depend on curvature properties of $S_{\lambda}$, see e.g.\ \cite{MR620265}, \cite{MR3608659}, \cite{MR4153099}. For the Laplacian $H_0=-\Delta$, i.e.\ $h_0(\xi)=\xi^2$, the surface $S_{\lambda}=\sqrt{\lambda}\mathbb{S}^{d-1}$ has everywhere nonvanishing Gauss curvature if $\lambda>0$. This fact lies at the heart of the Stein--Tomas theorem as well as the uniform resolvent estimates of Kenig--Ruiz--Sogge \cite{MR894584} that are behind the upper bound \eqref{LT bound} for $\gamma\leq 1/2$. We will prove generalizations of \eqref{Frank bound long range}, \eqref{DN higher dim} in Section \ref{Section upper bounds generalized kin. energies} for operators of the form~\eqref{gen. Schroedinger op.} 
(we actually allow $V$ to be a pseudodifferential operator). Our counterexamples show that these upper bounds are sharp. To simplify the exposition we state the result for the fractional Laplacian $H_0=(-\Delta)^s$. We remark that part (i) of the following theorem was already proved in \cite[Th. 6.1]{MR3608659} (see also \cite{MR3739326} for related resolvent estimates).

\begin{theorem}\label{theorem fractional}
Let $d\geq 1$, $s>0$ and $q\geq q_s$, where
\begin{align}\label{def. qs}
q_s:=
\begin{cases}
d/s\quad&\mbox{if }s<d,\\
1+\quad&\mbox{if }s=d,\\
1\quad&\mbox{if }s>d.
\end{cases}
\end{align}
Let $H_V=(-\Delta)^{s/2}+V$. Then then any eigenvalue $z\in\C$ of $H_V$ satisfies the following.
\begin{enumerate}
\item[{\rm(i)}] If $q\leq (d+1)/2$, then
\begin{align}\label{first estimate fractional}
|z|^{q-\frac{d}{s}}\leq D_{d,s,q} \|V\|_{q}^{q}.
\end{align}
\item[{\rm(ii)}] If $q>(d+1)/2$, then 
\begin{align}\label{second estimate fractional}
\dist(z,\R_+)^{q-\frac{d+1}{2}}|z|^{\frac{d+1}{2}-\frac{d}{s}}\leq D_{d,s,q} \|V\|_{q}^{q}.
\end{align}
\item[{\rm(iii)}] For any natural number $N$,
\begin{align*}
|z|^{\frac{d+1}{2}-\frac{d}{s}}\leq C_{d,s,N}\sup_{y\in\R^d}\int_{\R^d}(1+|\im z(x-y)|)^{-N}|V(x)|^{\frac{d+1}{2}}\rd x.
\end{align*}
\end{enumerate}
\end{theorem}
The estimate in (iii) corresponds to \eqref{DN higher dim}. Using explicit formulas for the resolvent kernel of the fractional Laplacian in terms of special functions one could probably replace the rapid decay by an exponential one. However, our proof only uses stationary phase estimates and works for more general constant coefficient operators $H_0$. In practice, the difference between \eqref{DN higher dim} and (iii) is not significant; only the decay scale $|\im z|^{-1}$ is. Observe that if $s<2d/(d+1)$ (this condition appears in \cite{MR3608659,MR3739326}), then one is always in the long-range case (ii) since $(d+1)/2<q_s$. The proof of (ii), (iii) could be obtained by closely following the arguments in \cite{MR3717979} and  \cite{MR4104544}, respectively. 
However, our main point here is to show that all the statements of Theorem 4 follow from the general results of Propositions \ref{thm universal bound 1 psdo} and \ref{thm universal bound 2 psdo} below in the special case $h_0(\xi)=|\xi|^s$. 

As a further consequence of our counterexample to the Laptev--Safronov conjecture, one can modify the construction of \cite[Th.~1]{MR3627408}, valid for $q>d$, to $q>(d+1)/2$.
 Here, $\sigma_{\rm p}(H_V)$ denotes the set of eigenvalues.
\begin{theorem}\label{thm accumulation}
Let $d\geq 2$, $q>(d+1)/2$ and $\epsilon>0$. 
Then there exists $ V\in L^{\infty}(\R^d)\cap L^q(\R^d)$ with $\max\{\|V\|_{\infty},\|V\|_q\}\leq \epsilon$ such that $\sigma_p(H_V)\backslash\R_+$ accumulates at every point in $\R_+$.
\end{theorem}

To conclude the introduction we give some comments on the idea behind the counterexample to the Laptev--Safronov conjecture.
A key difference to the constructions in \cite{MR3713021,MR2024415,MR4107520} (for embedded eigenvalues) is that the potential here is not explicit, but depends on an (unknown) eigenfunction of a compact operator $K$ (see the proof of Lemma \ref{lemma soft 1} for details). 
In \cite{MR3713021,MR2024415,MR4107520} one starts with a putative eigenfunction of $H_V$ and determines $V$ from the eigenvalue equation. The strategy we adopt here more closely follows the standard approach to prove upper bounds, the so-called Birman--Schwinger principle \cite{MR0142896}, \cite{Schwinger122}.
In its simplest form, this principle says that $z$ is an eigenvalue of $H_V$ if and only if $-1$ is an eigenvalue of the compact operator $\sqrt{|V|}(H_0-z)^{-1}\sqrt{V}$. A simplified sketch of the proof of \eqref{LT bound} for $d\geq 2$ and $\gamma\leq 1/2$ then goes as follows,
\begin{align}\label{BS principle}
1\leq \|\sqrt{|V|}(H_0-z)^{-1}\sqrt{V}\|\leq \|V\|_q\|(H_0-z)^{-1}\|_{p\to p'}\leq (D_{\gamma,d}|z|^{-\gamma})^{1/q}\|V\|_q,
\end{align}
where $p^{-1}+(p')^{-1}=q^{-1}$ and we recall that $q=\gamma+d/2$. The second inequality above is simply Hölder's inequality, while the last inequality follows from the work of Kenig--Ruiz--Sogge \cite{MR894584} (and is due to Frank \cite{MR2820160} in its rescaled version). This inequality is closely related to the Stein--Tomas theorem for the Fourier restriction operator $F_Sf:=\hat{f} \restrict S$, where $S=\sqrt{\lambda}\mathbb{S}^{d-1}$ if $z=\lambda+\I\epsilon$. The Knapp example shows that $p=2(d+1)/(d+3)$ (corresponding to $q=(d+1)/2$) is the best (largest) possible exponent in the inequality $\|\hat{f}\|_{L^2(\mathbb{S}^{d-1})}\leq C_p\|f\|_{L^p(\R^d)}$. The same is true for the $p\to p'$ resolvent estimate since 
\begin{align}\label{im part resolvent}
\im (H_0-(\lambda+\I\epsilon))^{-1}=-\frac{\epsilon}{(H_0-\lambda)^2+\epsilon^2}
\end{align}
and the right hand side converges to a constant times $F_S^*F_S$, as $\epsilon\to 0$. Similarly to the previous argument, the proof of \eqref{Frank bound long range} for $q>(d+1)/2$ ($\gamma>1/2$) in \cite{MR3717979} uses the non-uniform bound 
\begin{align*}
\|(H_0-z)^{-1}\|_{p\to p'}\lesssim \dist(z,\R_+)^{\frac{d+1}{2q}-1} \quad (|z|=1),
\end{align*}
which is also sharp \cite[Prop. 1.3]{MR4076079}. In our construction the potential $V$ is adapted to the Knapp example, making the second (Hölder) and third inequality in \eqref{BS principle} optimal simultanously. The only possible loss of optimality thus comes from the first inequality, and this may happen if the spectral radius of $\sqrt{|V|}(H_0-z)^{-1}\sqrt{V}$ is much smaller than its norm. We avoid this problem by working with \eqref{im part resolvent} instead of the full resolvent. It turns out that one can redefine $V$ (without making it larger in $L^q$ norm) in such a way that $z$ becomes an eigenvalue of $H_V$.

\medskip
\noindent
\textbf{Organization of the paper.}
In Section 2 we construct the counterexample to the Laptev--Safronov conjecture and prove Theorems 1--3 and 5. In Section 3 we give an alternative (non compactly supported) counterexample that is a perturbation of the Frank--Simon example for embedded eigenvalues. In Section 4 we prove an almost sharp quantitative lower bound on the norm of the compact operator $K$ (which implies an upper bound on the potential) and generalize the counterexample to generalized Schrödinger operators of the form \eqref{gen. Schroedinger op.}. Corresponding upper bounds for such operators (in particular, a proof of Theorem 4) are collected in Section 5.

\medskip
\noindent
\textbf{Notation.} For $a,b\geq 0$
the statement $a\lesssim b$ means that $a\leq C b$ for some universal
constant $C$. The expression $a\asymp b$ means $a\lesssim b$ and $b\lesssim a$.
 If the estimate depends on a parameter $\tau$, we indicate
this by writing $a\lesssim_{\tau} b$. In particular, if $\tau=N$, we always mean that the estimate is true for any natural number $N$, with a constant depending on $N$. The expression $a\lesssim b^{\kappa+}$ ($\kappa\in \R$) means $a\lesssim_{\delta}b^{\kappa+\delta}$ for any $\delta>0$, and similarly for $\kappa-$.
The dependence on the dimension and on other fixed quantities is always suppressed. An assumption $a\ll b$ means that there is a small constant $c$ such that if $a\leq cb$, then the ensuing conclusion holds. We also use $c$ as a generic positive constant in estimates involving exponentials, as in \eqref{DN higher dim}. The big oh notation $a=\mathcal{O}(b)$ means $|a|\lesssim b$ (here we are not assuming $a\geq 0$). For an integral operator $K$ on $\R^d$ we denote by $\|K\|_{r\to s}$ its $L^r\to L^s$ operator norm. If $r,s=2$, then we just write $\|K\|$. Similarly, we denote the $L^r$ norm of a function $f$ by $\|f\|_r$ and write $\|f\|$ if $r=2$. 
We denote by $\sigma(T)=\{z\in\C:\,T-z \text{ not boundedly invertible}\}$ the spectrum of a linear operator $T$. 
We write $\langle x\rangle=(1+x^2)^{1/2}$, where $x^2=x\cdot x$ for $x\in\R^d$. If nothing else is indicated, integrals are always understood to be over $\R^d$. The indicator function of a set $A$ is denoted by $\mathbf{1}_A$.
If we speak of a bump function we mean a smooth, compactly supported, real-valued function with values in $[0,1]$.

\section{Counterexample to the conjecture}

The counterexample to the Laptev--Safronov conjecture is based on Lemmas \ref{lemma soft 1} and \ref{lemma soft 2} below.
In the following we use the notation
$$\delta_{\lambda,\epsilon}(H_0):=\frac{\epsilon}{(H_0-\lambda)^2+\epsilon^2},$$ where $\lambda\in\R$ and $\epsilon>0$. We abbreviate this by $\delta_{\epsilon}(H_0)$ if $\lambda=1$.

Let $H_0=h_0(x,D)$ be a self-adjoint, elliptic pseudodifferential operator on $L^2(\R^d)$ with domain
\begin{align*}
    \dom{(H_0)}:=\set{u\in L^2(\R^d)}{H_0 u\in L^2(\R^d)}.
\end{align*}
Here $h(x,D)$ is the Kohn--Nirenberg quantization of a symbol $h\in S_{\rho,\delta}^m$ (the standard Hörmander classes, see e.g.\ \cite{MR2743652}) where $0\leq \delta<\rho\leq 1$ and $m>0$. Hence,
\begin{align*}
h_0(x,D)u(x)=(2\pi)^{-d}\int_{\R^d}\e^{\I x\cdot\xi}h(x,\xi)\hat{u}(\xi)\rd\xi,
\end{align*}
where $\hat{u}$ is the Fourier transform of a Schwartz function $u$. We will write $H_0\in OPS_{\rho,\delta}^m$. 
We assume that $H_0$ is elliptic, i.e.\ $|h_0(x,\xi)|\gtrsim |\xi|^m$ for $|\xi|\geq C$. By \cite[Prop. 5.5]{MR2743652}, we have $\dom{(H_0)}=H^{m}(\R^d)$. We also assume that $H_0$ is \emph{real}, i.e.\ commutes with complex conjugation. On the symbol level, this means that $\overline{h_0(x,\xi)}=h_0(x,-\xi)$. Note that this is the case for the Laplacian, for which $h_0(\xi)=\xi^2$. 

Let $U\subset\R^d$ be a nonempty open, precompact set,
and let $\chi=\mathbf{1}_{U}$ be its indicator function. In the following we will consider the operator $K:=\chi\delta_{\lambda,\epsilon}(H_0)\chi$.

\begin{lemma}
$K$ is compact.
\end{lemma}

\begin{proof}
Let $\Lambda=(1-\Delta)^{1/2}$ and write 
\begin{align*}
K=\chi\Lambda^{-m}(\Lambda^{m}\delta_{\lambda,\epsilon}(H_0)\Lambda^{m})\Lambda^{-m}\chi.
\end{align*}
By the Kato--Seiler--Simon inequality \cite[Theorem 4.1]{MR2154153}, $\chi\Lambda^{-m}$ is in the Schatten class $\mathfrak{S}^p$ for any $p>d/m$ (and $p\geq 2$). In particular, it is compact, and so is its adjoint $\Lambda^{-m}\chi$. It remains to show that $\Lambda^{m}\delta_{\lambda,\epsilon}(H_0)\Lambda^{m}$ is $L^2$ bounded. By Beal's theorem \cite[Theorem 3.2]{MR435933} it follows that $\delta_{\lambda,\epsilon}(H_0)\in OPS_{\rho,\delta}^{-2m}$, and by the $L^2$ boundedness of zero order pseudodifferential operators (see e.g.\ \cite[Theorem 5.3]{MR2743652} for the symbol classes considered here), $\Lambda^{m}\delta_{\lambda,\epsilon}(H_0)\Lambda^{m}$ is bounded.
\end{proof}

We next state an analog of the Birman--Schwinger principle for the operator $K$. The proof is a straightforward verification. Here we do not need to assume $\chi=\mathbf{1}_{U}$.

\begin{lemma}\label{bs principle}
Let $\mu\in\C\setminus \{0\}$. Then $\mu$ is an eigenvalue of $K$ if and only if the operator $(H_0-\lambda)^2+\epsilon^2-(\epsilon/\mu)\chi^2$ has nontrivial kernel. Moreover,
\begin{align*}
    \delta_{\lambda,\epsilon}(H_0)\chi:\ker(K-\mu)\to \ker((H_0-\lambda)^2+\epsilon^2-(\epsilon/\mu)\chi^2)
\end{align*}
is a linear isomorphism with inverse $\mu^{-1}\chi$.
\end{lemma}

The following lemma is a standard elliptic regularity result, but we provide a proof for completeness. 
\begin{lemma}\label{lemma elliptic regularity}
Let $\mu\in\C\setminus \{0\}$. Then $\ker(K-\mu)\subset C^{\infty}(U)$. 
\end{lemma}

\begin{proof}
By Lemma \ref{bs principle} it suffices to prove that if $u\in \dom{(H_0^2)}$,
\begin{align*}
    (H_0-\lambda)^2u+\epsilon^2u-(\epsilon/\mu)\chi^2u=0,
\end{align*}
then $u\in C^{\infty}(U)$. It is clear that the above equation takes the form $Pu=f$ with $P\in OPS^{2m}_{\rho,\delta}$ elliptic and $u,f\in L^2(\R^d)$. Let $Q \in OPS^{2m}_{\rho,\delta}$ be a parametrix for $P$ (see e.g.\ \cite[Ch. 1, Sect. 4]{MR2743652}). Then, modulo smooth functions, we have $QPu=u$ and hence $u=Qf\in H^{2m}(\R^d)$ by \cite[Prop. 5.5]{MR2743652}. To bootstrap this, we localize near a point $x_0\in U$ and let $\chi_j$ be a sequence of bump functions in $U$ such that $\chi_j=1$ near $x_0$ and $\chi_j\chi_{j-1}=\chi_j$. Then, again modulo smooth functions, $u_j=\chi_ju$ satisfies $P u_j=f_j$ with $f_j=[P,\chi_j]u=[P,\chi_j]u_{j-1}\in H^{j(\rho-\delta)}(\R^d)$ (again by \cite[Prop. 5.5]{MR2743652} and since the commutator reduces the order by $\rho-\delta$, see \cite[(3.24)]{MR2743652}). Applying the previous elliptic regularity estimate successively yields $u_j\in H^{2m+j(\rho-\delta)}(\R^d)$. By Sobolev embedding, $u_j\in C^{k}(\R^d)$ for $2m+j(\rho-\delta)>d/2+k$. This shows that $u$ is smooth at $x_0$.
\end{proof}

\begin{lemma}\label{lemma soft 1}
There exists $V\in L^{\infty}(\R^d)$ such that $z=\lambda+\I\epsilon$ is an eigenvalue of $H_V$ and $|V|\leq \|K\|^{-1}\chi$. 
\end{lemma}

\begin{proof}
Since $K$ is a nonnegative compact operator, its largest eigenvalue equals $\|K\|$. Hence, there is a nontrivial $\phi\in L^2(\R^d)$ such that $K\phi=\|K\|\phi$. Since $K$ is real we may and will assume that $\phi$ is real-valued. Using this together with the identity
\begin{align*}
\delta_{\lambda,\epsilon}(H_0)=\im (H_0-z)^{-1}=\frac{1}{2\I}((H_0-z)^{-1}-(H_0-\overline{z})^{-1}),
\end{align*}
the eigenvalue equation takes the form
\begin{align*}
(H_0-z)\psi=\|K\|^{-1}\chi\im\psi,
\end{align*}
where $\psi:=(H_0-z)^{-1}\phi$ and where we used that $\chi\phi=\phi$ since $\chi^2=\chi$. Here, $\im\psi$ denotes the imaginary part of a function, in distinction to the meaning of $\im$ above for the imaginary part of an operator. Let $\mathcal{N}:=\set{x\in U}{\psi(x)=0}$ be the nodal set of $\psi$, and set $V:=-\|K\|^{-1}\mathbf{1}_{U\setminus\mathcal{N}}\frac{\im\psi}{\psi}$. Note that the nodal set is well defined since $\phi$ is smooth in $U$, by Lemma~\ref{lemma elliptic regularity}, and hence $\psi$ is  smooth in $U$, by the pseudolocal property \cite[page 6]{MR2743652}. Here we are again using Beal's theorem to assert that $(H_0-z)^{-1}$ is a pseudodifferential operator. Then 
\begin{align*}
(H_0-z)\psi+V\psi=\|K\|^{-1}\mathbf{1}_{\mathcal{N}}\im\psi=0
\end{align*}
and $V$ satisfies the claimed bound.
\end{proof}

\begin{rem}
 In the case of the Laplacian, $H_0=-\Delta$, the set $U\setminus \mathcal{N}$ has positive Lebesgue measure. This can be most easily seen a posteriori: Since $\phi$ is nontrivial, 
 $\psi$ is nontrivial as well. If $U\setminus \mathcal{N}$ had zero measure, then we would have $-\Delta \psi = z \psi$, but this has no $H^2$ solution, as a consequence of Rellich's theorem \cite{MR0017816}.
\end{rem}
The $L^q$-norm of $V$ is estimated by
\begin{align}\label{Lp norm V}
\|V\|_q\leq  \|K\|^{-1}\|\chi\|_q,
\end{align}
so it remains to estimate $\|K\|$ from below and $\|\chi\|_q$ from above.

To avoid technicalities at this stage we restrict attention to the Laplacian $H_0=-\Delta$, i.e.\ $h_0(\xi)=\xi^2$. Without loss of generality (scaling) we restrict ourselves to $\lambda=1$. For $\epsilon>0$ let $T_{\epsilon}$ be an $\epsilon^{-1}\times\epsilon^{-1/2}$ tube centred at the origin, with long side pointing in the $x_1$ direction, i.e.\ 
\begin{align}\label{T_eps}
T_{\epsilon}=\{|x_1|< \epsilon^{-1},\,|x'|< \epsilon^{-1/2}\}.
\end{align}
Let $\chi_{\epsilon}$ be the indicator function of the tube $T_{\epsilon}$.
We then have the following (qualitative) lower bound for $K_{\epsilon}=\chi_{\epsilon}\delta_{\epsilon}(H_0)\chi_{\epsilon}$. In Lemma \ref{lemma lower bound K} below we will prove a quantitative (almost optimal) lower bound.

\begin{lemma}\label{lemma soft 2}
Let $H_0=-\Delta$, $\lambda=1$ and $\chi_{\epsilon}$ as above. For $0<\epsilon\ll 1$ the operator norm $\epsilon \|K_{\epsilon}\|$ is bounded below by a positive constant, independent of $\epsilon$.
\end{lemma}

\begin{proof}
We conjugate $K_{\epsilon}$ by $\e^{\I x_1}$ and rescale $(y_1,y')=(\epsilon x_1,\epsilon^{1/2}x')$. The resulting operator is isospectral to $\epsilon K_{\epsilon}$ and given by 
\begin{align}\label{Keps'}
K'_{\epsilon}=\chi_1\delta_{1}(H_{\epsilon}')\chi_1
\end{align}
where $\chi_1$ is the indicator function of $T_{1}$ and $H_{\epsilon}'=-2\I\partial_{y_1}-\Delta_{y'}-\epsilon\partial_{y_1}^2$. In the limit $\epsilon\to 0$ the operator $K'_{\epsilon}$ converges strongly to $K'_{0}$. Since this is not the zero operator, there exists an $L^2$-normalized function $f$ such that $\|K_0'f\|>0$. Now the strong convergence implies that for $0<\epsilon\ll 1$ we have $\|K_{\epsilon}'f\|\geq \|K_0'f\|/2$. This implies that
\begin{align*}
\epsilon\|K_{\epsilon}\|=\|K_{\epsilon}'\|\geq \|K_{\epsilon}'f\|\geq \|K_0'f\|/2
\end{align*}
and thus proves the claim.
\end{proof}
%

\begin{rem}
Laptev and Safronov based their conjecture on the famous Wigner--von Neumann example \cite{wigner1929merkwurdige} (see also \cite{MR0493421}), which is a potential decaying like $1/|x|$ with embedded eigenvalue at $\lambda=1$. The potential is in $L^q$ for any $q>d$, corresponding to $\gamma>d/2$ in \eqref{LT bound}. Had they been aware of the Ionescu--Jerison example \cite{MR2024415} they might have conjectured the smaller range $\gamma\leq 1/2$, for which \eqref{LT bound} indeed holds \cite{MR2820160}. To prove the weaker statement with $q>d$ in Theorem 1 one can replace the tube~ \eqref{T_eps} by the ball $\{|x|< \epsilon^{-1}\}$. Then the operator \eqref{Keps'} (under the scaling $y=\epsilon x$) is independent of $\epsilon$ and becomes
\begin{align*}
K'=\chi_1\delta_{1}(-\Delta_{y})\chi_1.
\end{align*}
One can repeat the argument in the proof of Lemma \ref{lemma soft 2} (i.e.\ $K'\neq 0$) and combine the result with Lemma \ref{lemma soft 1} to conclude Theorem 1 for $q>d$ since now $\|V_{\epsilon}\|_q\lesssim \epsilon^{1-d/q}$.
\end{rem}
\begin{rem}
A straightforward adaptation of the proof of Lemma \ref{lemma soft 2} yields the same conclusion for the Laplacian $H_0=-\Delta_g$ where $g$ is a short range perturbation of the Euclidean metric, $g_{ij}(x)-\delta_{ij}(x)=\mathcal{O}(|x|^{-2-})$ as $|x|\to\infty$. This shows that the results of Guillarmou, Hassell and Krupchyk \cite[Theorem 4]{MR4150258} are optimal for such metrics. The upper bounds in \cite{MR4150258} are proved for the more general case  of nontrapping asymptotically conic manifolds of dimension $d\geq 3$. Of course, scaling is not available in this situation, and the mentioned modification of Lemma \ref{lemma soft 2} only shows optimality for $\lambda$ in a bounded interval.
\end{rem}
%

\subsection{Proofs of Theorems 1--3}

Combining Lemmas \ref{lemma soft 1} and \ref{lemma soft 2} we obtain a sequence of potentials $V_{\epsilon}$ and a sequence $z_{\epsilon}=1+\I\epsilon$ of eigenvalues of $-\Delta+V_{\epsilon}$, $0<\epsilon\ll 1$, such that
\begin{align}\label{Lq bound Veps}
\|V_{\epsilon}\|_q\lesssim \epsilon^{1-\frac{d+1}{2q}}.
\end{align}
This follows from \eqref{Lp norm V} and the fact that $\|\chi_{\epsilon}\|_q=|T_{\epsilon}|^{1/q}$ (where $|\cdot|$ denotes Lebesgue measure). Theorems 1 and 2 follow immediately since $|z_{\epsilon}|\asymp 1$, $\dist(z,\R_+)=\epsilon$. 
Theorem 3 follows from the same counterexample.
We use in addition that, in the limit $\epsilon\to 0$, $\im\sqrt{z_\epsilon}/\epsilon\to 1/2$, and then the substitution $(x_1,x')=(\epsilon^{-1}y_1,\epsilon^{-1/2}y')$ yields
\begin{align*}
F_V(L\im\sqrt{z_\epsilon})&\lesssim \epsilon^{\frac{d+1}{2}}\int_{T_\epsilon}\exp(-L\im\sqrt{z_\epsilon}|x|) \rd x
\asymp \int_{T_1}\exp(-L|y_1|/2) \rd y\asymp L^{-1}.
\end{align*}

\subsection{Proof of Theorem 5}
As explained in \cite[Rem.~1]{MR3627408}, a counterexample to the Laptev--Safronov conjecture for a $q>(d+1)/2$ allows one to modify the construction in \cite[Th.~1]{MR3627408} to hold for this particular $q$. The only modification in the proof is to find a class of potentials satisfying the claim of \cite[Lem.~1]{MR3627408}, now for $q>(d+1)/2$, which is done in the following result.
\begin{lemma}\label{lem accum}
Let $d\geq 2$, $\lm\in \R_+$ and $q>(d+1)/2$.
For any $\eps_0,\delta_0,r_0>0$ there exists $V\in L^{\infty}(\R^d)\cap L^q(\R^d)$ with
$ \|V\|_{q}< \eps_0$, $\|V\|_{\infty}< \delta_0$ and such that there exists a non-real eigenvalue of $H_V$ in the ball  $B(\lambda,r_0)$.
\end{lemma}

\begin{proof}
If $\lambda=1$, we  use 
Lemmas \ref{lemma soft 1} and \ref{lemma soft 2} to obtain a sequence of potentials $V_{\epsilon}$ and a sequence $z_{\epsilon}=1+\I\epsilon$ of eigenvalues of $-\Delta+V_{\epsilon}$, $0<\epsilon\ll 1$, such that
$\|V_{\epsilon}\|_q\lesssim \epsilon^{1-\frac{d+1}{2q}}$ and $\|V_{\epsilon}\|_\infty\lesssim \epsilon$.
Then the claim follows by taking $\epsilon$ sufficiently small.
If $\lambda\neq 1$, the claim follows by scaling to the previous case.
\end{proof}

Now the proof of Theorem \ref{thm accumulation} is completely analogous to the one of \cite[Th.~1]{MR3627408}, using Lemma \ref{lem accum} instead of \cite[Lem.~1]{MR3627408}. Note that in \cite{MR3627408} the eigenvalues are constructed in the lower complex half-plane whereas here we constructed eigenvalues in the upper complex half-plane; one could take the adjoint operator to transform one case into the other one.

\section{Perturbation of embedded eigenvalues}

In this section we provide an alternative counterexample to the Laptev--Safronov conjecture that is closer to that suggested by Frank and Simon \cite{MR3713021}. We continue to make the same assumptions on $H_0$ and $U$ as in Lemma \ref{lemma soft 1}.


\begin{prop}\label{thm:embedded}
Let $\lambda\in\sigma(H_0)$ and $f\in\dom(H_0)$ with $\|f\|=1$. 
Assume that
$ \|(1-\chi)f\|\leq 1/4.$
Then, for any $\eps\geq 2\|(H_0-\lm)f\|$, there exists $V\in L^{\infty}(\R^d)$ such that $z=\lambda+i\eps$ is an eigenvalue of $H_0+V$  and $|V|\leq 4\eps \chi$.
In particular, if $\lambda$ is an eigenvalue and $(H_0-\lambda)f=0$, then the conclusion holds for all $\eps\geq 0$.
\end{prop}

\begin{proof}
We set $X=(H_0-\lambda)$ and write $\epsilon\delta_{\lambda,\epsilon}(H_0)=1-X^2/(X^2+\epsilon^2)$. By the Peter Paul inequality, $\epsilon\|X/(X^2+\epsilon^2)\|\leq 1/2$. Since $\|Xf\|\leq \epsilon/2$, this yields $\|X^2/(X^2+\epsilon^2)f\|\leq 1/4$, and hence (using $\|X^2/(X^2+\epsilon^2)\|\leq 1$ and $\|\chi\|\leq 1$)
\begin{align*}
\|\chi X^2/(X^2+\epsilon^2)\chi f\|\leq \|X^2/(X^2+\epsilon^2)f\|+\|X^2/(X^2+\epsilon^2)\|\|(1-\chi)f\|\leq \frac{1}{2}.
\end{align*}
We have also used $\|(1-\chi)f\|\leq 1/4$ in the last inequality. Using this once more, we obtain
\begin{align*}
\epsilon\|\chi \delta_{\lambda,\epsilon}(H_0)\chi f\|\geq \|\chi f\|-\frac{1}{2}\geq \frac{1}{4}.
\end{align*}
Thus $\epsilon\|K\|\geq 1/4$ and Lemma \ref{lemma soft 1} implies the claim.
\end{proof}

\begin{rem}
Recall that the self-adjointness of $H_0$ implies that for every $\lambda\in\sigma(H_0)$ there exists a normalized sequence $(f_n)_{n\in\N}\subset\dom(H_0)$ with $\|(H_0-\lambda)f_n\|\to 0$. Thus we can always find a normalized function for which $\|(H_0-\lambda)f\|$ is as small as we want. However, we need to make sure that $U$ is so large that the assumption $\|(1-\chi)f\|\leq 1/4$ is satisfied.
\end{rem}

\subsection{Alternative counterexample to Laptev--Safronov conjecture}

Let $q>(d+1)/2$.
Consider the sequence of potentials $V_n\in C^{\infty}(\R^d)$, $n\in\N$, in \cite[Theorem 2.1]{MR3713021} where $\lm=1$ is an embedded eigenvalue of $-\Delta+V_n$ for each $n$. The potentials satisfy $|V_n(x)|\lesssim(n+|x_1|+|x'|^2)^{-1}$.
In particular, $\|V_n\|_q\to 0$ and $\|V_n\|_{\infty}\to 0$ as $n\to\infty$.
Now fix $n\in\N$. Denote by $f_n$ a normalized eigenfunction corresponding to the embedded eigenvalue.
Let $U_n\subset \R^d$ be a compact subset that is so large that $\|(1-\mathbf{1}_{U_n})f_n\|\leq 1/4$.
Then, by Proposition~\ref{thm:embedded}, for all $\eps>0$ there exist potentials $W_{n,\eps}\in L^{\infty}(\R^d)$ such that $z_{\eps}=1+i\eps \in\sigma_p(-\Delta+V_n+W_{n,\eps})$ with 
$|W_{n,\eps}|\leq 4\eps\mathbf{1}_{U_n}$.
Let $(\eps_n)_n$ be such that $\epsilon_n=o(|U_n|^{-1/q})$ as $n\to\infty$.
Then
$$\|W_{n,\eps_n}\|_q \leq 4\eps_n |U_n|^{1/q}=o(1).$$
This disproves the Laptev-Safronov conjecture
since
$z_{\eps_n}\to 1$, $\|V_n+W_{n,\eps_n}\|_q\to 0$ in the limit $n\to\infty$.
Note that this construction is similar to the one in the previous section, as one can take $\epsilon_n=1/n$ and $U_n=T_{c_0/n}$ for a small positive constant~$c_0$. However, due to the additional $V_n$, the potentials here don't have compact supports.

\section{Quantitative lower bounds}

The aim of this section is to optimize the lower bound on $\epsilon \|K_{\epsilon}\|$ in Lemma~\ref{lemma soft 2}. Since the proof relied on soft arguments it did not provide a quantitative lower bound. The trivial upper bound is $\epsilon \|K_{\epsilon}\|\leq 1$. By stretching the tube $T_{\epsilon}$ in \eqref{T_eps}, we are able to prove an almost sharp lower bound. More precisely, let $\chi_{\epsilon}$ be the indicator function of $T_{\epsilon/M}$, where $M\gg 1$. 

\begin{lemma}\label{lemma lower bound K}
Let $\chi_{\epsilon}$ be as above. Then
\begin{align}\label{lower bound K}
\epsilon\|K_{\epsilon}\|\geq 1-\mathcal{O}(M^{-2+}).
\end{align}
\end{lemma}

\begin{proof}
To prove the lower bound in \eqref{lower bound K} we first write 
\begin{align*}
K_{\epsilon}=\delta_{\epsilon}(H_0)-(1-\chi_{\epsilon})\delta_{\epsilon}(H_0)-\chi_{\epsilon}\delta_{\epsilon}(H_0)(1-\chi_{\epsilon}).
\end{align*}
We will treat $\delta_{\epsilon}(H_0)$ as a main term and the other terms as errors. Since $\|\chi_{\epsilon}\|_{\infty}\leq 1$,
\begin{align*}
\|K_{\epsilon}f\|\geq \|\delta_{\epsilon}(H_0)f\|-\|(1-\chi_{\epsilon})\delta_{\epsilon}(H_0)f\|-\|\delta_{\epsilon}(H_0)(1-\chi_{\epsilon})f\|
\end{align*}
for any $f\in L^2(\R^d)$. We take 
\begin{align*}
\hat{f}_{\epsilon}(\xi):=\eta_0((c_0\epsilon)^{-1}(\xi_1-1),(c_0\epsilon)^{-1/2}\xi'),
\end{align*}
where $\xi=(\xi_1,\xi')\in \R\times\R^{d-1}$, $\eta_0\in C_0^{\infty}(B(0,2))$ is a nonnegative bump function equal to $1$ on $B(0,1)$, and $c_0$ is a small positive constant, to be chosen later.  We mention that $f$ is known as a 'Knapp example' in harmonic analysis, see e.g.\ \cite[Example 1.8]{MR3971577}. For $\xi$ in the support of $\hat{f}_{\epsilon}$, we have $|\xi^2-1|=\mathcal{O}(c_0\epsilon)$, which means that $\epsilon\delta_{\epsilon}(\xi)\geq 1-\mathcal{O}(c_0^2)$ there. Here and in the following $\delta_{\epsilon}(\xi)=\frac{\epsilon}{(\xi^2-1)^2+\epsilon^2}$. By Plancherel, we conclude that
\begin{align}\label{l.b. delta f}
\epsilon\|\delta_{\epsilon}(H_0) f_{\epsilon}\|\geq (1-\mathcal{O}(c_0^2))\|f_{\epsilon}\|.
\end{align}
Note that $\epsilon\delta_{\epsilon}(\xi)\leq 1$ and is smooth on the scale of $f_{\epsilon}$; in other words, we can write
\begin{align*}
(\delta_{\epsilon} \hat{f}_{\epsilon})(\xi)=\epsilon^{-1}\eta((c_0\epsilon)^{-1}(\xi_1-1),(c_0\epsilon)^{-1/2}\xi')
\end{align*}
for some bump function $\eta$, similar to $\eta_0$. More precisely, by $\eta$ we really mean a family of such bump functions, with smooth norms bounded uniformly in~$\epsilon$. Thus, both $f_{\epsilon}$ and $\delta_{\epsilon}(H_0) f_{\epsilon}$ are Schwartz functions decaying rapidly away from $T_{c_0\epsilon}$; in particular,
\begin{align*}
\|(1-\chi_{\epsilon}) f_{\epsilon}\|\lesssim_N (c_0M)^{-N}\|f_{\epsilon}\|,\quad
\|(1-\chi_{\epsilon})\delta_{\epsilon}(H_0) f_{\epsilon}\|\lesssim_N\epsilon^{-1}(c_0M)^{-N}\|f_{\epsilon}\|,
\end{align*}
where we used that $\|f_{\epsilon}\|^2\asymp (c_0\epsilon)^{\frac{d+1}{2}}$ and $\chi_{\epsilon}=1$ on $T_{\epsilon/M}$. Together with \eqref{l.b. delta f} and the fact that $\epsilon\delta_{\epsilon}(\xi)\leq 1$, this yields
\begin{align*}
\epsilon\|K_{\epsilon}f_{\epsilon}\|\geq(1-\mathcal{O}(c_0^2)-\mathcal{O}_N((c_0M)^{-N}))\|f_{\epsilon}\|.
\end{align*}
Choosing $c_0=M^{-1+}$ and taking $N$ sufficiently large yields the lower bound in~\eqref{lower bound K}.
\end{proof}

\begin{rem}
In view of $\epsilon\|K_{\epsilon}\|\leq 1$ the bound \eqref{lower bound K} is optimal in the limit $M\to\infty$. If we choose e.g.\ $M=\log(1/\epsilon)$, then we obtain from \eqref{lower bound K} that for all $q>(d+1)/2$,
\begin{align*}
\lim_{\epsilon\to 0}\epsilon\|K_{\epsilon}\|=1,\quad \lim_{\epsilon\to 0}\|V_{\epsilon}\|_q=0.
\end{align*}
\end{rem}

\subsection{Lower bounds for constant-coefficient operators}
Here we supplement Lemma \ref{lemma soft 2} with quantitative bounds. We consider more general constant-coefficient operators than the Laplacian, that is we allow $H_0=h_0(D)$. For instance, for the fractional Laplacian (Theorem 4) we have $h_0(\xi)=|\xi|^s$. For the lower bound we only assume that $h_0$ is a tempered distribution, smooth in a neighborhood of some $\xi^0\in\R^d$ and such that $\lambda:=h_0(\xi^0)$ is a regular value of $h_0$. If the hypersurface $\{h_0(\xi)=\lambda\}$ has everywhere nonvanishing Gauss curvature, then local versions of  \eqref{LT bound}, \eqref{Frank bound long range} hold (see Section \ref{Section upper bounds generalized kin. energies}). If the Gauss curvature vanishes at some point, then the upper bound will be worse, but the following example still provides a lower bound. One could improve the example if more is known about the local geometry of the isoenergy surface.

The following example is very close to that of Lemma \ref{lemma lower bound K} and is based on the factorization
\begin{align*}
h_0(\xi)-\lambda=e(\xi)(\xi_1-a(\xi')),
\end{align*}
which holds locally near $\xi^0$, with $e$ nonvanishing there. We suppress the dependence of $e,a$ on $\lambda$.
By a linear change of coordinates we assume, as we may, that $a(0)=0$ and $\partial_{\xi'}a(0)=0$. Then $a(\xi')=\mathcal{O}(|\xi'|^2)$. We take a Knapp example $f$ whose Fourier support is contained in the cap
\begin{align*}
\kappa_{\epsilon}:=\{|\xi_1|<c_0\epsilon, |\xi'|<(c_0\epsilon)^{1/2}\}.
\end{align*}    
Clearly, $\kappa_{\epsilon}$ is contained in an $\epsilon$ neighborhood of the isoenergy surface $\{\xi_1=a(\xi')\}$, hence 
\begin{align*}
\epsilon\delta_{\epsilon}(\xi)=\frac{\epsilon^2}{(h_0(\xi)-\lambda)^2+\epsilon^2}\geq 1-\mathcal{O}(c_0^2).
\end{align*}
Now the same argument as for the Laplacian shows that \eqref{lower bound K} holds, for exactly the same function $\chi_{\epsilon}$.

%

\section{Upper bounds}\label{Section upper bounds generalized kin. energies}

In this section we prove a generalization of the bounds  \eqref{LT bound}, \eqref{Frank bound long range} for Schrödinger type operators of the form
\begin{align}\label{more gen. Schroedinger op.}
H_V=h_0(D)+V(x,D).
\end{align}
We first consider the classical Schrödinger operator $H_V=-\Delta+V(x)$ to explain what types of estimates we will prove. By homogeneity, the estimates \eqref{LT bound}, \eqref{Frank bound long range} may by reduced to $|z|=1$ and by elliptic regularity to a small neighborhood of $z=1$. Hence these bounds can collectively be expressed as
\begin{align}\label{universal bound 1}
\dist(z,\sigma(H_0))^{(q-(d+1)/2)_+}\lesssim \|V\|_q^q
\end{align}   
for $q> d/2$ (or $q=1$ if $d=1$), while \eqref{DN higher dim} reads as
\begin{align}\label{universal bound 2}
1\lesssim\sup_{y\in\R^d}\int_{\R^d}|V(x)|^{\frac{d+1}{2}}\exp(-c|\im z||x-y|)\rd x
\end{align}
for some constant $c>0$ (we used that $\im\sqrt{z}\asymp |\im z|$ for $|z-1|$ small). Here, $\sigma(H_0)$ denotes the spectrum of $H_0$. The bounds \eqref{universal bound 1}--\eqref{universal bound 2} are \emph{universal} in the sense that they are essentially independent of the specific form of $h_0$, in a sense that we will make precise now. 

In the following, we always assume that the spectral parameter is $z=\lambda+\I\epsilon$, with $\lambda\in\sigma(H_0)$ and $|\epsilon|\leq 1$. Then $\dist(z,\sigma(H_0))=|\epsilon|$. One could use the Phragm\'en-Lindel\"of maximum principle to extend the results to the region $|\epsilon|\geq 1$ (see e.g.\ \cite[Appendix A]{MR3608659}, \cite{MR4150258}, \cite{Ruiz2002}), but we will not pursue this.
 
We assume that $h_0$ is a tempered distribution that is smooth near the foliation (see \eqref{Fermi surface} for the definition of $S_{\lambda}$)
\begin{align}\label{foliation}
S=\bigcup_{\lambda\in I}S_{\lambda}=\set{\xi\in\R^d}{h_0(\xi)\in I}.
\end{align}
Here $I\subset\R$ is a fixed compact subset of the set
\begin{align*}
\set{\lambda\in \R}{\nabla h_0(\xi)\neq 0 \mbox{ for all } \xi\in\R^d \mbox{ such that } h_0(\xi)=\lambda}
\end{align*}
of regular values of $h_0$. We assume that $S_{\lambda}$ is compact and has everywhere non-vanishing curvature for each $\lambda\in I$. The following lemma was proved e.g.\ in \cite[Lemma 3.3]{MR3608659}. It is closely related to the Stein--Tomas theorem for the Fourier restriction operator.
%
\begin{lemma}\label{lemma ST}
Let $\eta$ be a bump function. Then
\begin{align*}
\sup_{\stackrel{\lambda\in I}{|\epsilon|\leq 1}}\|\eta(D)[h_0(D)-(\lambda+\I\epsilon)]^{-1}\|_{p\to p'}\lesssim 1.
\end{align*}
\end{lemma}
The setup \eqref{more gen. Schroedinger op.} 
generalizes \eqref{gen. Schroedinger op.} 
in that we allow $V(x,D)$ to be a pseudodifferential operator. We assume that its Kohn--Nirenberg symbol $V(x,\xi)$ is smooth in the fibre variable $\xi$, but we don't assume smoothness in $x$. More precisely, assume that
\begin{align}\label{symbol class}
C_{q,\Omega,N}(V):=\sum_{|\alpha|\leq N}\sup_{\xi\in\Omega}\|\partial_{\xi}^{\alpha}V(\cdot,\xi)\|_q<\infty
\end{align}
for some sufficiently large $N$ ($N>d$ would suffice) and some pre-compact subset $\Omega\subset\R^d$ such that $S\Subset \Omega$. 
The condition~\eqref{symbol class} is the natural generalization of $V\in L^{q}(\R^d)$ and reduces to the latter if $V$ is a potential. The order of $V(x,\xi)$ will not play a role here since we will always localize in Fourier space. In fact, we place ourselves in the following abstract setting:
Let $\eta,\overline{\eta}$ be bump functions supported on $\Omega$ such that $\eta^2+\overline{\eta}^2=1$ in a neighborhood of $S$. Define
\begin{align*}
H_V^{\eta}&:=h_0(D)+\eta(D) V\eta(D): \Ran P\to \Ran P,\\
H_V^{\overline{\eta}}&:=h_0(D)+\overline{\eta}(D)V\overline{\eta}(D): \Ran\overline{P}\to\Ran\overline{P},
\end{align*}
where $P=1_{\Omega}(D)$, $\overline{P}=1-P$.
In order to avoid imposing global conditions on $h_0$ and $V$ in $\xi$ we make the following assumption:
\begin{align}\label{black box}
I\subset \rho(H_V^{\overline{\eta}}),\quad c_{q,I}:=\sup_{\lambda\in I,\,|\epsilon|\leq 1}\|[H_V^{\overline{\eta}}-(\lambda+\I\epsilon)]^{-1}\|_{p\to p'}<\infty.
\end{align}
Here $p$ is uniquely determined by $q^{-1}=p^{-1}-(p')^{-1}$, $\rho(\cdot)$ denotes the resolvent set and $\|\cdot\|_{p\to p'}$ denotes the $L^{p}\to L^{p'}$ norm. In most applications, \eqref{black box} can easily be proved by standard elliptic estimates (see Subsection \ref{subsect. Proof of Theorem 4} for an example). This usually requires that $q\geq q_0$ for some $q_0\geq 1$ depending on $h_0$. We will ignore this and simply use \eqref{black box} as a black box assumption. Hence, we only assume that $q\geq 1$ in the following. In the proofs of Theorems \ref{thm universal bound 1 psdo}--\ref{thm universal bound 2 psdo} below we will make use of the \textit{smooth Feshbach-Schur map} \cite{MR2409163}. The latter is defined by $(H_V-z)\mapsto F_{\eta}(z)$, where
\begin{align*}
F_{\eta}(z):=H_V^{\eta}-z-\eta V\overline{\eta}
[H_V^{\overline{\eta}}-z]^{-1}\overline{\eta} V\eta: \Ran P\to \Ran P.
\end{align*} 
Here and in the following we abbreviate $\eta=\eta(D)$ and $\overline{\eta}=\overline{\eta}(D)$.
Theorem 1 in \cite{MR2409163} (with $V=\Ran P$ there) asserts that $H_V-z$ is invertible if and only if $F_{\eta}(z)$ is invertible, and that
\begin{align*}
[H_V-z]^{-1}=Q[F_{\eta}(z)]^{-1}Q^{\#}+\overline{\eta}[H_V^{\overline{\eta}}-z]^{-1}\overline{\eta},
\end{align*}
where
\begin{align*}
Q&:=\eta-\overline{\eta}[H_V^{\overline{\eta}}-z]^{-1}\overline{\eta}V\eta: \Ran P\to L^2(\R^d),\\
\quad
Q^{\#}&:=\eta-\eta V\overline{\eta}[H_V^{\overline{\eta}}-z]^{-1}\overline{\eta}: L^2(\R^d)\to \Ran P.
\end{align*}

\begin{prop}\label{thm universal bound 1 psdo}
Assume $H_V$ is as above and that \eqref{symbol class}, \eqref{black box} hold. Then
\begin{align}\label{universal bound 1 psdo}
\dist(z,\sigma(H_0))^{(q-(d+1)/2)_+}\lesssim C_{q,\Omega,N}(V)^q
\end{align} 
holds for every eigenvalue $z=\lambda+\I\epsilon$ of $H_V$, $\lambda\in I$, $|\epsilon|\leq 1$, with implicit constant depending on $h_0,d,q,I,|\Omega|$, but not on $z$, $V$.
\end{prop}

\begin{proof}
By compactness of $I$, it suffices to prove \eqref{universal bound 1 psdo} at $\re z=\lambda$ for a single $\lambda\in I$.
We first consider the case $q\leq (d+1)/2$.
Then \eqref{universal bound 1 psdo} is equivalent to the statement that if $C_{q,\Omega,N}(V)$ is sufficiently small, then $z$ is not an eigenvalue of $H_V$. We will show that $H_V-z$ is invertible using the smooth Feshbach-Schur map.
We write 
\begin{align}\label{Feta}
F_{\eta}(z)=h_0(D)-z+\eta\widetilde{V}_z\eta,\quad
\widetilde{V}_z:= V- V\overline{\eta}[H_V^{\overline{\eta}}-z]^{-1}\overline{\eta} V.
\end{align}
By Lemma \ref{lemma Hoelder} below and \eqref{black box},
\begin{align*}
\|\eta\widetilde{V}_z\|_{p'\to p}&\lesssim  \|\eta V\|_{p'\to p}+\|[H_V^{\overline{\eta}}-z]^{-1}\|_{p\to p'}\|\eta V\|_{p'\to p}\|\overline{\eta}V\|_{p'\to p}\\
&\lesssim_{|\Omega|} C_{q,\Omega,N}(V)+c_{q,I} C_{q,\Omega,N}(V)^2=:D_{I,q,\Omega,N},
\end{align*}
where the $L^p$ boundedness of $\eta,\overline{\eta}$ is a consequence of the Mikhlin multiplier theorem \cite[Th. 5.2.7]{MR3243734}. 
By Lemma \ref{lemma ST} we conclude that
\begin{align*}
\|\eta\widetilde{V}_z\eta[h_0(D)-z]^{-1}\|_{p\to p}\leq \|\eta\widetilde{V}_z\|_{p'\to p}\|\eta[h_0(D)-z]^{-1}\|_{p\to p'}\lesssim D_{I,q,\Omega,N},
\end{align*}
and hence, by a geometric series argument, that $F_{\eta}(z)$ is boundedly invertible in $L^p(\R^d)$ if $C_{q,\Omega,N}(V)\ll 1$. Since the spectrum is independent of $p$ (see e.g.\ \cite[Th. 14.3.10]{MR2359869}\footnote{The statement there is given in one dimension. However, the proof only uses general facts about $L^p$ spaces, valid in any dimension.}) it follows that $0$ is not in the $L^2$ spectrum of $F_{\eta}(z)$, or equivalently, that $z$ is not in the spectrum of $H_V$.

To prove the claim for $q>(d+1)/2$ we follow \cite{MR3717979} and interpolate the bound of Lemma \ref{lemma ST} (for $q=(d+1)/2$) with the trivial estimate
\begin{align*}
\|[h_0(D)-z]^{-1}\|_{2\to 2}\leq \dist(z,\sigma(H_0))^{-1},
\end{align*}
which yields
\begin{align*}
\|\eta[h_0(D)-z]^{-1}\|_{p\to p'}\lesssim |\epsilon|^{\frac{d+1}{2q}-1}.
\end{align*}
Repeating the previous argument, we find that $z$ cannot be an eigenvalue if the quantity $b:=D_{I,q,\Omega,N}|\epsilon|^{\frac{d+1}{2q}-1}$ is too small; in other words, if $z$ is an eigenvalue, then we must have $b\gtrsim 1$. If we set $a:=C_{q,\Omega,N}(V)$, then this means that we must have
$|\epsilon|^{1-\frac{d+1}{2q}}\lesssim a+c_{q,I} a^2$.
Since we are assuming $|\epsilon|\leq 1$, this is always satisfied if $a\geq 1/c_{q,I}$. If $a\leq 1/c_{q,I}$, then the condition becomes $|\epsilon|^{1-\frac{d+1}{2q}}\lesssim a$, in which case \eqref{universal bound 1 psdo} holds.
\end{proof}

In the above proof we used the following generalization of Hölder's inequality.

\begin{lemma} \label{lemma Hoelder}
Assume that $V$ satisfies \eqref{symbol class} and $\eta$ is a bump function supported on $\Omega$. Then, for $N$ sufficiently large,
\begin{align*}
\|V(x,D)\eta(D)\|_{p'\to p}+\|\eta(D)V(x,D)\|_{p'\to p}\lesssim |\Omega| C_{q,\Omega,N}(V)
\end{align*}
whenever $q^{-1}=p^{-1}-(p')^{-1}$.
\end{lemma}

\begin{proof}
By duality it suffices to estimate the first summand on the left. The kernel of $V(x,D)\eta(D)$ (recall that we use the Kohn--Nirenberg quantization) is given by
\begin{align*}
k(x,x-y)=(2\pi)^{-d}\int_{\R^d}\e^{\I (x-y)\cdot\xi}V(x,\xi)\eta(\xi)\rd\xi.
\end{align*}
Due to the cutoff $\eta$ the integral is restricted to $\xi\in \Omega$. Integration by parts shows that 
\begin{align*}
|k(x,u)|\lesssim_N \langle u\rangle^{-N}\int_{\Omega}\sum_{|\alpha|\leq N}|\partial_{\xi}^{\alpha}V(x,\xi)|\rd\xi,
\end{align*} 
and \eqref{symbol class}, together with Minkowski's inequality, then provides the estimate
\begin{align}\label{IBP bound k(x,x-y)}
\|k(x,u)\|_{L^q_x}\lesssim_N \langle u\rangle^{-N} |\Omega|C_{q,\Omega,N}(V).
\end{align}
Changing variables from $y$ to $u=x-y$ and using Minkowski's and Hölder's inequality, we get
\begin{align*}
&\|\int k(x,x-y)f(y)\rd y\|_{L^p_x}\leq \int\| k(x,u)f(x-u)\|_{L^p_x}\rd u
\leq \|f\|_{p'} \int \|k(x,u)\|_{L^q_x}\rd u
\end{align*}
and \eqref{IBP bound k(x,x-y)} yields the claimed inequality.

\end{proof}

\begin{rem}
If $V$ is a potential, then of course $\|V\eta(D)\|_{p'\to p}\lesssim \|V\|_q$, i.e.\ the factor $|\Omega|$ can be dispensed with. From this point of view, a more natural norm than \eqref{symbol class} would e.g.\ be 
\begin{align*}
\sup_{\xi\in\Omega}\|V(\cdot,\xi)\|_q+\sum_{1\leq |\alpha|\leq N}\int_{\Omega}\|\partial_{\xi}^{\alpha}V(\cdot,\xi)\|_q\, \rd \xi.
\end{align*}
However, we view $\Omega$ as fixed, which justifies our use of the simpler norm \eqref{symbol class}.
\end{rem}

We now establish a more precise version of Lemma \ref{lemma ST}. In the following we denote $R^{\eta,\zeta}_{\lambda,\epsilon}=\eta(D)[h_0(D)-(\lambda+\I\epsilon)]^{-\zeta}$ and also write $R^{\eta,\zeta}_{\lambda,\epsilon}(x-y)$ for its kernel.

\begin{lemma}\label{lemma DN}
Let $\eta$ be a bump function and $\zeta\in\C$, $0\leq \re\zeta\leq (d+1)/2$, $\epsilon\in [-1,1]$. Then we have the kernel bound
\begin{align}\label{decay Green}
\sup_{\lambda\in I}|R^{\eta,\zeta}_{\lambda,\epsilon}(x)|\lesssim_{N}\e^{C|\im\zeta|^2} \langle x\rangle^{-\frac{d-1}{2}+\re \zeta}\langle\epsilon x\rangle^{-N}
\end{align}
for some $C>0$.
\end{lemma}

\begin{proof} 

Again, it suffices to prove this for a fixed $\lambda$.
We absorb $\lambda$ into the symbol, i.e.\ we consider $p(\xi)=h_0(\xi)-\lambda$. By a partition of unity and a linear change of coordinates we may assume that, locally near an arbitrary point of $\Omega$, either $p\neq 0$ or $\partial p/\partial \xi_1>0$. In the case $p\neq 0$ we get the stronger bound
\begin{align*}
|R^{\eta,\zeta}_{\lambda,\epsilon}(x)|\lesssim_N \langle x-y\rangle^{-N}.
\end{align*}

We turn to the case $\partial p/\partial \xi_1>0$ and consider $\zeta=1$ first (i.e.\ the resolvent).
By the implicit function theorem, $\{p(\xi)=0\}$ is then the graph of a smooth function $\xi_1=a(\xi')$, and we have the factorization
\begin{align}\label{factorization 1}
(p(\xi)-i\epsilon)^{-1}=(\xi_1-a(\xi')-i\epsilon q(\xi))^{-1}q(\xi)
\end{align}
where $q(\xi)=(\xi_1-a(\xi'))/p(\xi)>0$, see e.g.\  \cite[Section 14.2]{MR705278}, \cite[Lemma 3.3]{MR3608659} and \cite[Section 3.1]{MR4009459}. There it is sufficient to work with the limiting distributions corresponding to $\epsilon=0\pm$, which would yield \eqref{decay Green} in this case. Here we need to keep $\epsilon$ fixed to get the desired decay for nonzero $\epsilon$. In the following we assume that $\epsilon>0$; the case $\epsilon<0$ is similar.
The factorization \eqref{factorization 1} does not work well for this since $q$ depends on $\xi_1$. To remedy this problem, we follow the approach of Koch--Tataru \cite{MR2094851}, albeit in the much simpler setting of constant coefficients. Lemma 3.8 in \cite{MR2094851} provides the alternative factorization
\begin{align}\label{factorization 2}
e(\xi)(p(\xi)-i\epsilon)=\xi_1+a(\xi')+i\epsilon b(\xi'),
\end{align}
where $e$ is elliptic ($e\neq 0$) and $a,b$ are real-valued. This is a version of the Malgrange preparation theorem \cite[Th. 7.5.5]{MR1065993} or the classical Weierstrass preparation theorem \cite[Th. 7.5.1]{MR1065993} in the analytic case. We appeal to \cite{MR2094851} because it makes the dependence on $\epsilon$ explicit. Note that the imaginary part $b$ is now independent of $\xi_1$. The symbols $e,a,b$ can be found by iteratively solving a system of algebraic equations and using Borel resummation of the resulting formal series (see \cite[Lemma 3.9 and 3.10]{MR2094851}). Moreover, $a,b$ have asymptotic expansions in powers of $\epsilon$, while $e$ has an asymptotic expansions in powers of $\epsilon$ and $\xi_1$. We will only need the first term $b_1$ in the expansion of $b$. Changing variables $\xi\to \xi_1+a(\xi')$ we are reduced to $p(\xi)=\xi_1+i\epsilon q(\xi)$ for some real-valued function $q$. By the proof of \cite[Lemma 3.9]{MR2094851} we have $b_1=1/(1+q_1^2)$, where $q_1=\partial_{\xi_1}q|_{\xi_1=0}$. Therefore, $b\geq c$ on the closure of $\Omega$ for some constant $c>0$ (we used compactness and the smallness of $\epsilon$). Since we have constant coefficients, the simple parametrix (5.5) in \cite{MR2094851}, with (operator-valued) kernel $K(x_1-y_1)$, given by 
\begin{align}\label{parametrix}
K(x_1)=\mathbf{1}_{x_1<0}\,\e^{\epsilon x_1b(D')}e^{-ix_1 a(D')},
\end{align}
is exact, i.e.\ $(D_1+a(D')+\I \epsilon b(D'))K$ is the identity (we denote both the operator and the kernel by $K$ here). By the stationary phase estimate (for complex-valued phase functions) \cite[Th. 7.7.5]{MR1065993},
\begin{align*}
|K(x)|\lesssim \langle x\rangle^{-\frac{d-1}{2}} \e^{-c|\epsilon x|}+\mathcal{O}_N(\langle x\rangle^{-N}),
\end{align*}
Using the factorization \eqref{factorization 2} and extending $1/e$ globally as a Schwartz function, we obtain \eqref{decay Green} in the case $\zeta=1$. The case $\zeta\neq 1$ requires only minor modifications. The kernel in \eqref{parametrix}
is replaced by
\begin{align*}
K_{\zeta}(x_1)=\chi_-^{\zeta-1}(x_1)\e^{\epsilon x_1 b(D')}e^{-ix_1 a(D')},
\end{align*}
where $\chi_{-}^{w}(\tau):=1_{\tau<0}|\tau|^{w}/\Gamma(w+1)$, $w\in\C$, where $\Gamma$ is the usual Gamma function.
Then $(D_1+a(D')+\I \epsilon b(D'))^{-\zeta}K_{\zeta}$ is the identity. This follows immediately by applying the inverse Fourier transformation to the following identity (see \cite{MR1065993}, specifically the explanation after Example 7.1.17)
\begin{align*}
\mathcal{F}\left(\tau\mapsto \e^{-\delta \tau}\chi_+^{\zeta}(\tau\right)(\xi)=\e^{-i\pi(\zeta+1)/2}(\xi-\I\delta)^{-\zeta-1},\quad \delta>0,\quad \zeta\in\C.
\end{align*}
Again, by stationary phase,
\begin{align*}
|K_{\zeta}(x)|\lesssim \e^{C|\im\zeta|^2}(\langle x\rangle^{-\frac{d-1}{2}+\re \zeta} \e^{-c|\epsilon x|}+\mathcal{O}_N(\langle x\rangle^{-N}))
\end{align*}
for $0\leq \re\zeta\leq (d+1)/2$. The growth estimate in $|\im\zeta|$ comes from a standard estimate on the Gamma function (see e.g.\ \cite[Appendix A.7]{MR3243734}).
\end{proof}

To state an analog of the estimate \eqref{universal bound 2} for $H_V$ in \eqref{more gen. Schroedinger op.} 
we assume
\begin{align}\label{symbol class 2}
C_{\epsilon, q,\Omega,N}(V):=\sum_{|\alpha|\leq N}\sup_{\xi\in\Omega}\sup_{y\in\R^d}\|\langle\epsilon (x-y)\rangle^{-N}\partial_{\xi}^{\alpha}V(x,\xi)\|_{L^q_x}<\infty
\end{align}
for some sufficiently large $N$ (again, $N>d$ would work). 
The norm \eqref{symbol class 2} is the analog of the right hand side of~\eqref{universal bound 2}. We also replace \eqref{black box} by the new black box assumption
\begin{align}\label{black box 2}
I\subset \rho(H_V^{\overline{\eta}}),\quad c_{q,I,V}:=\sup_{\lambda\in I,\,|\epsilon|\leq 1}\|\eta V\overline{\eta}[H_V^{\overline{\eta}}-(\lambda+\I\epsilon)]^{-1}\|_{p\to p}<\infty,
\end{align}
where we recall that $q^{-1}=p^{-1}-(p')^{-1}$.  
Note that, in contrast to \eqref{black box}, the potential still appears in \eqref{black box 2} and thus we need a $p\to p$ norm here. In many applications of interest (for instance, in the proof of Theorem 4), \eqref{black box 2} can be estimated perturbatively in terms of $H_0^{\overline{\eta}}$, with an effective constant $c_{q,I,V}$ in \eqref{black box 2}, i.e. a constant only depending on $C_{\epsilon, q,\Omega,N}$, but not on $V$ itself (see Subsection \ref{subsect. Proof of Theorem 4}).

\begin{prop}\label{thm universal bound 2 psdo}
Assume that \eqref{symbol class 2}, \eqref{black box 2} hold for some $q\leq (d+1)/2$. Then every eigenvalue $z=\lambda+\I\epsilon$ of $H_V$, $\lambda\in I$, $|\epsilon|\leq 1$, satisfies
\begin{align}\label{universal bound 2 psdo}
1\lesssim C_{\epsilon,q,\Omega,N}(V)
\end{align} 
with implicit constant depending on $h_0,d,q,I,|\Omega|$, but not on $z$, $V$.
\end{prop}

\begin{proof}
We again use the smooth Feshbach--Schur map. Thus the claim \eqref{universal bound 2 psdo} is equivalent to the statement
\begin{align*}
C_{\epsilon,q,\Omega,N}(V)\ll_{\lambda} 1\implies F_{\eta}(z)\mbox{ boundedly invertible},
\end{align*}
where $F_{\eta}(z)$ is given by \eqref{Feta}.
Again, by $p$-independence of the spectrum it suffices to prove invertibility in $L^p$, with $q^{-1}=p^{-1}-(p')^{-1}$, and this would follow (by geometric series) from 
\begin{align*}
\|\eta\widetilde{V}_z\eta[h_0(D)-z]^{-1}\|_{p\to p}<1,
\end{align*}
and this in turn would follow from 
\begin{align}\label{to prove universal bound 2 psdo}
\|V\eta[h_0(D)-z]^{-1}\|_{p\to p}\lesssim C_{\epsilon,q,\Omega,N}(V)
\end{align}
since then, by \eqref{Feta}, \eqref{black box 2}, \eqref{to prove universal bound 2 psdo},
\begin{align*}
\|\eta\widetilde{V}_z\eta[h_0(D)-z]^{-1}\|_{p\to p}&\lesssim (1+\|\eta V\overline{\eta}[H_V^{\overline{\eta}}-z]^{-1}\overline{\eta}\|_{p\to p})C_{\epsilon,q,\Omega,N}(V)\\
&\lesssim (1+c_{q,I,V})C_{\epsilon,q,\Omega,N}(V).
\end{align*}
To estimate \eqref{to prove universal bound 2 psdo}, let $k(x,x-y)$ be the kernel of $V\eta_1$ (where $\eta_1$ is a bump function like $\eta$, but with $\eta_1\eta=\eta$) 
and let $R_{\lambda,\epsilon}^{\eta}(x-y)$ be the kernel of $\eta[h_0(D)-z]^{-1}$. Then the kernel of $K:= V\eta[h_0(D)-z]^{-1}$ is 
\begin{align*}
K(x,y)=\int_{\R^d}k(x,x-u)R_{\lambda,\epsilon}^{\eta}(u-y)\rd u.
\end{align*}
As a warmup, we consider first the easiest case where $d=1$ and $V$ is a potential. Then $k(x,x-y)=V(x)\delta(x-y)$, and Lemma \ref{lemma DN} (with $\zeta=1$) yields
\begin{align}\label{K 1-1 norm}
\|K\|_{1\to 1}\leq \sup_y\int_{\R^d}|K(x,y)|\rd x\lesssim_{N} \sup_y\int_{\R^d}|V(x)|\langle\epsilon(x-y)\rangle^{-N}\rd x.
\end{align}
Comparing to \eqref{symbol class 2}, the right hand side is bounded by $C_{\epsilon,1,\Omega,N}(V)$, and hence if the latter is small, then $\|K\|_{1\to 1}<1$. When $V$ is no longer required to be a potential (but still in $d=1$), then the previous estimate is replaced by
\begin{align}\label{K 1-1 norm bis}
\sup_y\int_{\R^d}|K(x,y)|\rd x\lesssim \sup_y\int\int |k(x,u)|\langle\epsilon(x-u-y)\rangle^{-N}  \rd x \rd u,
\end{align}
where we first used the change of variables $u\to x-u$ and then Fubini. We insert $1=\langle u\rangle^{N}\langle u\rangle^{-N}$ and estimate the double integral by
\begin{align*}
C_N\sup_{y,u} \int_{\R^d}|k(x,u)|\langle u\rangle^N\langle\epsilon(x-u-y)\rangle^{-N}  \rd x
\end{align*}
where $C_N=\int \langle u\rangle^{-N}\rd u$. Then \eqref{IBP bound k(x,x-y)} and \eqref{symbol class 2}, together with the first inequality in~\eqref{K 1-1 norm}, yield $\|K\|_{1\to 1}\lesssim C_{\epsilon,1,\Omega,N}(V)$. Moving on to the general, higher-dimensional case, we use Stein interpolation on the analytic family $K_{\zeta}:=V^{\zeta}\eta[h_0(D)-z]^{-\zeta}$ to prove~\eqref{to prove universal bound 2 psdo}. For $\re\zeta=0$, we have the trivial bound
\begin{align*}
\|K_{\zeta}\|_{2\to 2}\lesssim \e^{c|\im\zeta|}\quad (\re\zeta=0).
\end{align*}
For $\re\zeta=q$, we use the estimate of Lemma \ref{lemma DN} to get \eqref{K 1-1 norm bis} for $K_{\zeta}$, i.e.
\begin{align*}
\|K_{\zeta}\|_{1\to 1}\lesssim \e^{c|\im\zeta|}C_{\epsilon,q,\Omega,N}(V)^q\quad (\re\zeta=q).
\end{align*}
Interpolating the last two estimates gives $\|K_1\|_{p\to p}\lesssim C_{\epsilon,q,\Omega,N}(V)$, which is just \eqref{to prove universal bound 2 psdo}, i.e.\ what we needed to prove.
\end{proof}

\subsection{Proof of Theorem 4}\label{subsect. Proof of Theorem 4}
By scaling, it suffices to prove the bounds for $|z|=1$ only. We only give a proof in the case $\re z\gtrsim 1$, which is the most difficult one. Note that the factor $\im z(x-y)$ is dimensionless as is should be. Hence we can take $I=[1-\delta,1+\delta]$ in \eqref{foliation} and find that $S=\{1-\delta'\leq |\xi|\leq 1+\delta'\}$ for some small positive constants $\delta,\delta'$. Recall that $\Omega\subset\R^d$ was chosen such that $S\Subset \Omega$ (see the paragraph after \eqref{symbol class}). This implies that $|h_0(\xi)-z|\geq C_{\Omega}+\langle\xi\rangle^s$ for all $\xi\in\R^d\setminus\Omega$ and hence, by Sobolev embedding, $$\|[H_0^{\overline{\eta}}-z]^{-1}\|_{p\to p'}\lesssim \|[H_0^{\overline{\eta}}-z]^{-1}\|_{H^{-\frac{\sigma}{2}}\to H^{\frac{\sigma}{2}}}\lesssim C_{\Omega}^{\frac{\sigma}{s}-1}$$ for $\sigma/d\geq p^{-1}-(p')^{-1}=q^{-1}$, where we denoted the $L^2$ based Sobolev space of order $\sigma$ by $H^{\sigma}$ and used Plancherel in the last inequality.  
We choose $\sigma=d/q_s$. 
If $C_{\Omega}^{\frac{\sigma}{s}-1}\|V\|_q \ll 1$, then \eqref{black box} holds by Hölder's inequality and a geometric series argument, and hence Proposition~\ref{thm universal bound 1 psdo} yields (i), (ii).

Moving on to the proof of (iii), the claim would follow from Proposition~\ref{thm universal bound 2 psdo} if we could show \eqref{black box 2}. For brevity, we restrict our attention to the case $s<d$. Precisely, we will show that $c_{q,I,V}<\infty$ if  
\begin{align}\label{asspt V Bessel}
\sup_{y\in\R^d}\int_{\R^d}\langle x-y\rangle^{-N}|V(x)|^{q}\rd x\ll 1,
\end{align}
where $N\gg 1$ and $q\geq d/s$.
Let us abbreviate the constant $c_{q,I,V}$ by $c_V$. We also set $c_0:=\|V\overline{\eta}[H_V^{\overline{\eta}}-z]^{-1}\|_{p\to p}$ (by compactness we can fix $z$). Without loss of generality assume that the inverse Fourier transform of $\eta$ is normalized in $L^1$, so that $\eta(D)$ is an isometry in $L^p$ (and similarly for $\overline{\eta}$). If we could prove $c_0<1$, then a geometric series argument would yield $c\leq c_0/(1-c_0)$ and we would be done. The next lemma establishes $c_0<1$.

\begin{lemma}
Let $s<d$. If \eqref{asspt V Bessel} holds with $q\geq d/s$, then $c_0<1$.
\end{lemma}

\begin{proof}
By the Mikhlin multiplier theorem \cite[Th. 5.2.7]{MR3243734} it suffices to prove this for $\|V\Lambda^{-s}\|_{p\to p}$ in place of $c_0$, where we recall that $\Lambda=(1-\Delta)$. Standard estimates for Bessel potentials (see e.g.\ \cite[Prop. 6.1.5]{MR3243741}) yield $\Lambda^{-s}(x-y)\lesssim |x-y|^{s-d}\e^{-|x-y|/2}$. Clearly, we may bound the exponential from above by $\langle x-y\rangle^{-N}$ for any $N$, which we will do. In view of the elementary estimate
\begin{align}\label{sum of indicator fcts}
\langle x-y\rangle^{-N}\lesssim \sum_{j=0}^{\infty}2^{-Nj}\mathbf{1}_{|x-y|\leq 2^{j}}
\end{align}
it would suffice to prove the following bound on 
$c_j:=\|V\Lambda_{s,j}\|_{p\to p}$, where $\Lambda_{s,j}$ has kernel $|x-y|^{s-d}\mathbf{1}_{|x-y|\leq 2^j}$:
\begin{align*}
c_j\lesssim 2^{j(d+s-d/q)}\sup_{u\in\R^d}\|V\|_{L^q(B(u,2^{j+1}))}.
\end{align*}
By homogeneity it suffices to prove this for $j=0$. Using $|V(x)|\leq \sum_{u\in \Z^d}|V_u(x)|$, with $V_u(x)=V(x)\mathbf{1}_{|x-u|\leq 2}$, we estimate
\begin{align*}
|\langle V\Lambda_{s,0} f,g\rangle |\leq \sum_{u\in \Z^d}\int\int |V_u(x)||g_u(x)||f(y)||x-y|^{s-d}\mathbf{1}_{|x-y|\leq 1}\rd x\rd y.
\end{align*}
Note that, by the triangle inequality, we can insert $\mathbf{1}_{|y-u|\leq 3}$ for free into the integral. Then, by Young's inequality (or by the Hardy--Littlewood--Sobolev inequality if $q=d/s$) and Hölder (once for integrals and once for sums),
\begin{align*}
|\langle V\Lambda_{s,0} f,g\rangle |&\leq \sum_{u\in \Z^d}\|V\|_{L^q(B(u,2))}\|f\|_{L^p(B(u,3))}\|g\|_{L^{p'}(B(u,2))}\\
&\leq \sup_{u\in\Z^d}\|V\|_{L^q(B(u,2))}\big (\sum_{u\in \Z^d}\|f\|_{L^p(B(u,3))}^p\big)^{1/p}\big(\sum_{u\in \Z^d}\|g\|_{L^{p'}(B(u,2))}^{p'}\big)^{1/p'}\\
&\lesssim \sup_{u\in\R^d}\|V\|_{L^q(B(u,2))}\|f\|_p\|g\|_{p'}.
\end{align*}
This completes the proof.
\end{proof}

\subsection*{Acknowledgements} The authors wish to thank Ari Laptev and Rupert Frank for many illuminating discussions.


\end{document}